\documentclass{amsart}

\usepackage{amsmath,amsthm,amssymb}
\usepackage{graphicx,color}
\usepackage{hyperref}
\usepackage{mathrsfs}
\usepackage[utf8]{inputenc}
\usepackage{wasysym}

\definecolor{Red}{cmyk}{0,1,1,0}

\definecolor{Blue}{cmyk}{1,1,0,0}

\theoremstyle{plain}
\newtheorem{theorem}{Theorem}[section]
\newtheorem{corollary}[theorem]{Corollary}
\newtheorem{proposition}[theorem]{Proposition}

\newtheorem{conjecture}[theorem]{Conjecture}
\newtheorem*{theorem-main}{Main Results}

\theoremstyle{definition}
\newtheorem{definition}[theorem]{Definition}
\newtheorem{remark}[theorem]{Remark}
\newtheorem{example}[theorem]{Example}

\title[{T.F.} for Topological Markov chains]
{Thermodynamic Formalism for Topological Markov Chains on Standard Borel Spaces}
\author[L. Cioletti]{L. Cioletti}
\address{Departamento de Matemática, Universidade de Bras\'ilia, 70910-900, Bras\'ilia, Brazil}
\email{cioletti@mat.unb.br} 
\author[E. A. Silva]{E. A. Silva}
\address{Departamento de Matemática, Universidade de Bras\'ilia, 70910-900, Bras\'ilia, Brazil}
\email{e.a.silva@mat.unb.br} 
\author[M. Stadlbauer]{M. Stadlbauer}
\address{Departamento de Matem\'atica, Universidade Federal do Rio de Janeiro, 21941-909 \phantom{Rio}Rio de Janeiro (RJ), Brazil.}
\email{manuel@im.ufrj.br}
\thanks{M. Stadlbauer is supported by FAPERJ and CNPq and L. Cioletti is supported by CNPq}
\date{}

\begin{document}

\vspace{-7pt}

%

\maketitle
	
\begin{quote}
	\footnotesize
	\textsc{Abstract.} 
	We develop a Thermodynamic Formalism for 
	bounded continuous potentials defined on the sequence space 
	$X\equiv E^{\mathbb{N}}$, where $E$ is a general standard Borel space. 
	In particular, we introduce meaningful concepts of entropy 
	and pressure for shifts acting on $X$ and obtain the existence 
	of equilibrium states as finitely additive probability measures for any bounded continuous potential. 
	Furthermore, we establish convexity and other structural properties 
	of the set of equilibrium states, prove a version of the 
	Perron-Frobenius-Ruelle theorem under additional assumptions 
	on the regularity of the potential
	and show that the Yosida-Hewitt decomposition of these
	equilibrium states does not have a purely finite additive part. 
	
	We then apply our results to the construction of invariant measures
	of time-homogeneous Markov chains taking values on a general Borel
	standard space and obtain exponential asymptotic stability for a class of
	Markov operators. We also construct conformal measures 
	for an infinite collection of interacting random paths which are
	associated to a potential depending on infinitely many coordinates.
	Under an additional differentiability hypothesis, we show how this 
	process is related after a proper scaling limit to a 
	certain infinite-dimensional diffusion.
\end{quote}
\vspace{21pt}

\section{introduction}

One of the principal motivations of Ergodic Theory is to 
understand the statistical behavior of a 
deterministic dynamical system $T:X\to X$ by studying 
invariant probability measures of the system. 
In this context, ergodic theorems provide quantitative information 
on the asymptotic behavior of typical orbits of $T$.
However, if $T$ is continuous, $X$ is compact and the dynamical system 
has some sort of mixing behavior, then there  exists a  
plethora of these invariant measures. In these cases, the 
theory of Thermodynamic Formalism is nowadays a 
recognized method for making a canonical choice of an invariant measure. 
That is, one fixes a continuous potential $f:X\to\mathbb{R}$ 
which encodes some qualitative behaviour of the system 
and considers those invariant probability measures, 
the so-called \textit{equilibrium states}, 
which satisfy a certain variational problem with respect this 
potential and which exist by compactness.  
However, if $X$ is not a compact space, additional  
hypotheses on $f$ are required in order to ensure the existence of such
canonical ergodic probability measures.

For example, if $X$ is a shift-invariant, closed subset of 
 $ {E}^{\mathbb{N}}$, where $E$
is an infinite countable set, the existence of 
equilibrium states is non-trivial and 
has been intensively investigated, due to its applications 
to the Gauss map, to partially hyperbolic dynamical systems 
and to unbounded spin systems in Statistical Mechanics 
on  one-dimensional one-sided lattices.  
From the viewpoint of abstract Thermodynamic Formalism, 
Mauldin \& Urbanski, Sarig and many others 
developed a rather complete theory for potentials on $X\subset {E}^{\mathbb{N}}$, where $E$
is an infinite countable alphabet (
see, for example, 
\cite{MR2018604,MR2551790,MR1738951,MR1822107,MR1818392,MR1955261,sarig2009lecture,MauldinUrbanski:1999,MauldinUrbanski:2001,MauldinUrbanski:2003}).
From the viewpoint of Statistical Mechanics, there is also a vast literature about
unbounded spins systems, they could be either the set of integer numbers $\mathbb{Z}$ or 
continuous $\mathbb{R}$, but usually the interactions are unbounded as in SOS, discrete Gaussian, $\Phi^4$ models and so on.
See, for example, 
\cite{MR1704666,MR0471115,MR735977,MR2807681,MR3025428,MR0446251,MR3735628,MR1370101,MR1115800}.
Even though in all these references, the concepts of pressure, entropy and thermodynamic 
limit play a major role, we do not yet have a unified framework which
relates these concepts across areas. 
For example, the potentials considered in the 
Thermodynamic Formalism literature usually
depend on infinitely many coordinates (which can be seen as infinite-body interactions) and satisfy  
suitable regularity conditions, and the alphabet is countable. 
On the other hand, in the Statistical Mechanics literature, the 
potential is typically less regular, infinite-range, the potential might be translation invariant or not, some times quasi-periodic potentials are considered  and the spins might take values in an uncountable set (or uncountable alphabet), 
but are usually defined in terms of finite-body interactions.  
The theory developed in this article now allows to consider potentials given by infinite-body interactions, uncountable alphabets 
and general bounded and continuous potentials. 
These three theories are related as shown
in the following diagram. We shall remark that no proper inclusion on the diagram 
below is possible. 

\begin{figure}[h]
	\centering
	\includegraphics[width=0.65\linewidth]{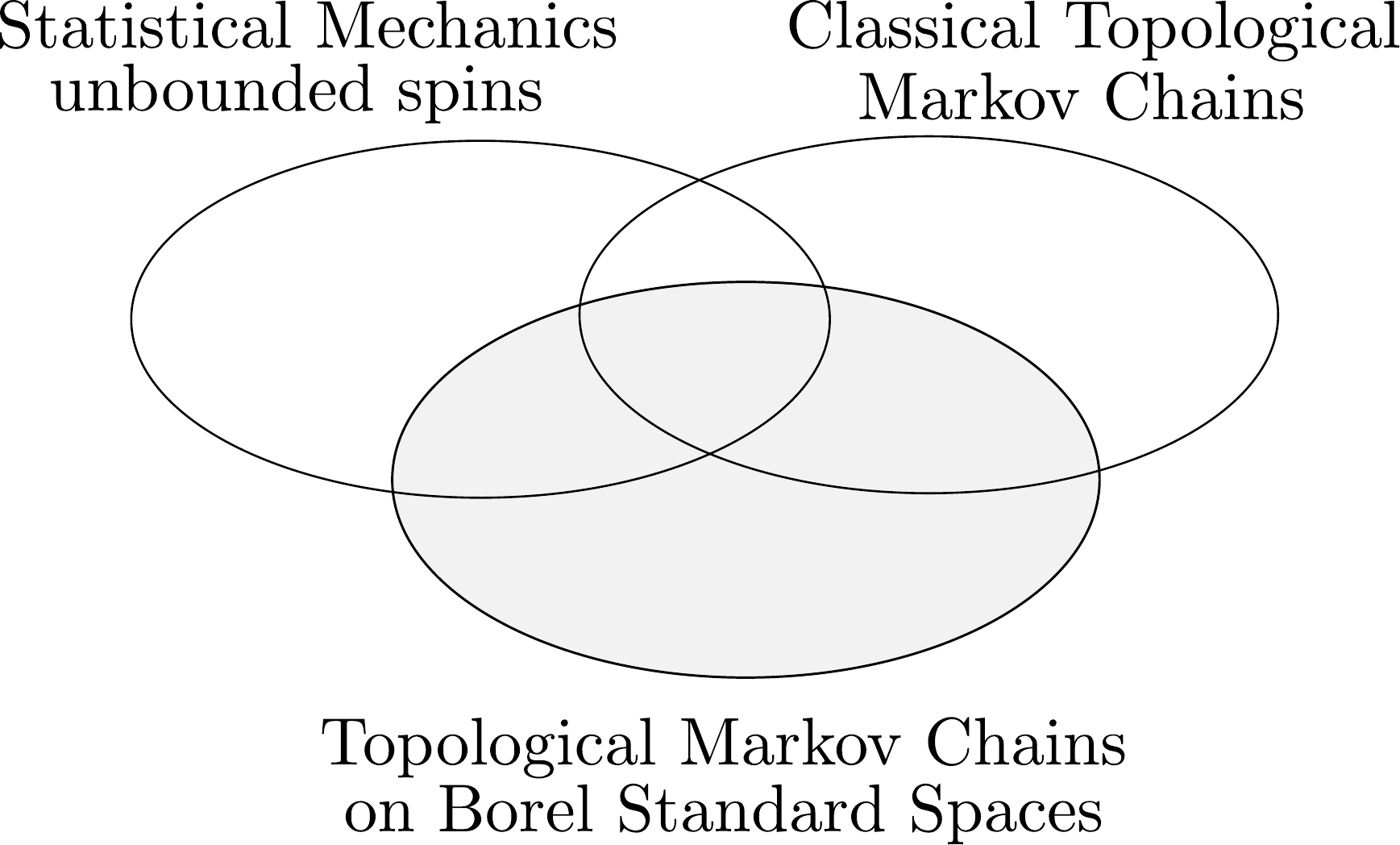}
	\caption[Comparison]{Relations among the three theories.}
	\label{fig:diagrama-3t}
\end{figure}

The classical Thermodynamic Formalism has its starting point in the seminal work by 
David Ruelle \cite{MR0234697} on the lattice gas model and was subsequently developed 
for subshifts of finite type, which are subsets of  
$M^{\mathbb{N}}$ with $M=\{1,\ldots, k\}$  (see, for example, 
\cite{MR1793194,MR2423393,MR1085356,MR2129258}) and are nowadays well-known 
tools in the context of hyperbolic dynamical systems. 
By considering a notion of pressure based on local returns, the Gurevic pressure, 
Sarig was able to extend the principles of Thermodynamic Formalism 
in \cite{MR1738951} to countable alphabets and obtained, among other things, 
a classification of the underlying dynamics into positively recurrent, 
null recurrent or transient behaviour through convergence of the transfer 
operator in \cite{MR1818392} or a proof of Katok's conjecture on the 
growth of periodic points of surface diffeomorphisms (\cite{Sarig:2013}).

However, from the viewpoint of Statistical Mechanics, it is also of interest to consider shift spaces with  
a compact metric alphabet which was done, for example, in 
\cite{ACR17,MR2864625,CL-rcontinuas-2016,MR3538412,CER17,MR3377291,MR3194082,2017arXiv170709072S}. In  \cite{MR2864625}, 
a Ruelle operator formalism was developed for the alphabet 
$M=\mathbb{S}^1$ and extended to general compact metric alphabets in \cite{MR3377291}.
As uncountable alphabets do not fit in the classical theory, as the number of preimages 
under the shift map is uncountable, the authors considered an \textit{a priori} measure $\mu$
defined on $M$ which allows to define a generalized Ruelle operator and prove a Perron-Frobenius-Ruelle Theorem.  
We would like to point out that the use of an \textit{a priori} measure is a standard procedure in 
Equilibrium Statistical Mechanics in order to deal with continuous spin systems, 
see \cite{MR2807681,MR1241537}, and, in combination with the given potential 
function, is also closely related to the notion of a transition kernel from probability theory. 

In this setting, it is necessary to propose new concepts of 
entropy, the so-called \emph{variational entropy}, and pressure. 
So an \emph{equilibrium state} for a continuous potential 
$f$ is an element of $\mathscr{M}_{\sigma}(X)$, the
set of all shift-invariant Borel probability measures, such that this 
measure realizes the supremum 
\[
\sup_{\mu\in \mathscr{M}_{\sigma}(X)}\{\mathtt{h}^{\mathtt{v}}(\mu)+ \langle\mu,f\rangle\},
\]
where $\mathtt{h}^{\mathtt{v}}(\mu)$ is the variational entropy 
of $\mu$ as introduced in \cite{MR3377291}. The associated 
variational principle was obtained in \cite{MR3377291} and the 
uniqueness of the equilibrium state in the class of Walters potentials in \cite{ACR17}. 
In there, the authors also  showed that 
the variational entropy defined in \cite{MR3377291} equals the 
\emph{specific entropy} commonly used in  
Statistical Mechanics (see \cite{MR2807681}). 
As a corollary, a variational formulation for the specific 
entropy is derived. 
It is also worth noting that several results for countable 
alphabets can be recovered by choosing a suitable a priori measure on the 
one-point compactification of $\mathbb{N}$ (see \cite{MR3377291}) 
and that the concepts of Gibbs measures and equilibrium states are 
equivalent if one considers potentials which are H\"older continuous 
or in Walters' class (\cite{BFV-ETDS-2018,MR3690296,CL-rcontinuas-2016,MR2861437}).

The aim of this article is to develop a Thermodynamic Formalism for
continuous and bounded potentials and alphabets which are standard Borel spaces. 
In this very general setting, one has 
to consider \textit{ergodic finitely additive probability measures} 
instead of ergodic probability measures as it will turn out in 
Theorem \ref{teo-exist-equi} and Corollary \ref{cor:unique-additive-measure} 
below that the following holds.

\medskip 
\noindent\textbf{Main Results}.
\textbf{(Equilibrium States)}. Let $f$ be a bounded and continuous potential. 
Then there exists a shift invariant and finitely additive 
measure which attains the supremum
\[\sup_{\mu\in \mathscr{M}^{a}_{\sigma}(X)} \mathrm{h}^{\mathrm{v}}(\mu) +\langle\mu,f\rangle.\] 
\textbf{(Ergodic Optimization)}. Let $E$ be a non-compact space, then there exists a bounded and continuous 
potential $f$, having a unique maximizing measure 
\[
m(f)
=
\sup_{\mu \in \mathscr{M}_\sigma^a(X)} \langle\mu,f\rangle
\] 
which is finitely but not necessarily  countably additive.  
\medskip 

Although finitely additive measures lead to a very abstract setting, we shall mention that 
these objects have been for a long time important mathematical objects in several
branches of pure and applied Mathematics, 
and naturally occur, for example, in the Fundamental Theorem of Asset Pricing
under the absence of arbitrages of the first kind (\cite{MR2732838}).

\bigskip 

This paper is organized as follows.
In Section \ref{sec-preliminares} we introduce the basic notation and 
recall the definition of the space $rba(X)$ as well as some of its basic properties. 
After that, the Ruelle operator acting on $C_{b}(X,\mathbb{R})$ 
is introduced, where $X = E^\mathbb{N}$
is a cartesian product of a general standard Borel space $E$.
In Section \ref{sec-RPF} we prove a Perron-Frobenius-Ruelle (PFR) theorem for bounded
H\"older potentials defined on $X$ and obtain a Central Limit Theorem as a corollary. 
Thereafter, we use 
PFR theorem to motivate the definition of the entropy and pressure. 
This leads to a natural definition of an equilibrium state as an element 
of $rba(X)$. We prove its existence for general bounded continuous 
potentials and also show that the supremum in the variational problem 
is attained by some shift-invariant regular finitely additive Borel probability 
measure. As a complement, it is proven that the set of equilibrium states is convex and 
compact and that bounded H\"older potentials admit equilibrium states whose 
Yosida-Hewitt decomposition does not have a purely finitely additive part.  
In Section \ref{sec:extreme_positive}, we then prove a characterization 
of the extremal mesures in order to obtain the second part of our  main theorem. 
Thereafter, in Section \ref{sec:applications}, the above 
Perron-Frobenius-Ruelle theorem is applied in the context of ergodic optimization 
and asymptotic stability of stochastic processes, and,  
we show in part \ref{sec-applications} how to use this 
theorem in order to construct an equilibrium state for 
infinite interacting random paths subject to an infinite-range potential.
We briefly discuss how their scaling limits are connected 
to some diffusions in infinite dimensions.

\section{Preliminaries}\label{sec-preliminares}

A measurable space $(E,\mathscr{E})$ is a standard Borel space 
if there exists a metric $d_E$ such that $(E,d_E)$ is a complete
separable metric space and $\mathscr{E}$ is the 
Borel sigma-algebra. 
Good examples to have in mind in order to compare our results with 
the classical ones in the literature are  
a finite set $\{1,\ldots, d\}$, the set of positive integers $\mathbb{N}$, a 
compact metric space $K$ or
the Euclidean space $\mathbb{R}^d$.
Throughout this paper, $X$ denotes the product space $E^{\mathbb{N}}$ and 
$\sigma:X\to X$, $(x_1,x_2,\ldots) \mapsto (x_2,x_3,\ldots)$ is the left shift. 
The space $X$ is regarded as a
metric space with  metric 
\[
d_{X}(x,y)
=
\sum_{n=1}^{\infty} \frac{1}{2^n}\min\{d_{E}(x_n,y_n),1\}.
\]
As easily can be verified, $X$ is always a bounded, complete and 
separable metric space, even though it may not be compact.
Furthermore, we refer to $C_{b}(X,\mathbb{R})$ as the Banach space of all 
real-valued bounded continuous functions endowed with its standard supremum norm.

\bigskip 

A Borel finitely additive signed measure on a topological space $X=(X,\tau)$
is an extended real valued set-function 
$\mu:\mathscr{B}(\tau)\to \mathbb{R}\cup\{-\infty,+\infty\}$ 
which satisfies
\textit{(i)} $\mu$ assumes at most one of the values $-\infty$ and $\infty$, 
\textit{(ii)} $\mu(\emptyset) = 0$,
\textit{(iii)} for each finite family $\{A_1,\ldots,A_n\}$
of pairwise disjoint sets in $\mathscr{B}(\tau)$, we have  
$\mu(A_1\cup\ldots\cup A_n)=\mu(A_1)+\ldots+\mu(A_n)$.
If $\sup_{A\in\mathscr{B}(\tau)}|\mu(A)|<+\infty$ for all $A\in \mathscr{B}(\tau)$, then
we say that $\mu$ is bounded. A Borel finitely additive signed measure $\mu$ is called regular
if for any $A\in \mathscr{B}(\tau)$ and $\varepsilon>0$, 
there exists a closed set $F\subset A$ and an open set $O\supset A$ 
such that for all Borel sets $C\subset O\setminus F$ we have 
$|\mu(C)|<\varepsilon$.
The total variation norm of 
a Borel finitely additive signed measure $\mu$ 
is defined by 
\[
\|\mu\|_{TV}\equiv 
\sup
\left\{
\sum_{k=1}^n |\mu(A_k)|: \ \{A_1,\ldots, A_n\}\subset\mathscr{B}(\tau)\ \text{is a partition of}\ X
\right\}.
\] 
It is known that the space of all regular bounded Borel finitely additive signed measures on a 
topological space $X$ endowed with the total variation norm is a Banach space and that,
since $X$ is a metric space,
the topological dual $C_{b}(X,\mathbb{R})^{*}$ is isometrically isomorphic 
to $(rba(X),\|\cdot\|_{TV})$ (see IV - Th. 6 in \cite{MR0117523} or
Th. 14.9 in \cite{MR2378491}). By \cite[p. 261]{MR0117523},
every $f\in C_{b}(X,\mathbb{R})$ is integrable with respect to every
$\mu\in rba(X)$, and its integral will be denoted by 
either $\mu(f)$, $\int_{X}f\, d\mu$ or $\langle \mu,f \rangle$. 
A countably additive Borel measure is an element $\mu\in rba(X)$ which is both
countably additive and non-negative, that is, 
$\mu(A)\geq 0$ for all $A\in \mathscr{B}(\tau)$.
If, in addition, $\mu(X)=1$ then $\mu$ is called a countably additive Borel 
probability measure, and we will make use of $\mathscr{M}_1(X)$ for the subset of $rba(X)$ of all 
countably additive Borel probability measures.
Furthermore, a regular finitely additive bounded Borel signed measure is said to be
shift-invariant if $\mu(f)=\mu(f\circ \sigma)$ for all $f\in C_{b}(X,\mathbb{R})$.

\bigskip 

In this paper, a generalized version of the Ruelle transfer operator 
will play a major role. Therefore, we first fix a 
Borel probability measure $p$ on $E$ and a potential $f\in C_{b}(X,\mathbb{R})$. 
The Ruelle operator is  defined as the positive linear operator 
$\mathscr{L}_{f}:C_{b}(X,\mathbb{R})\to C_{b}(X,\mathbb{R})$ 
sending $\varphi\longmapsto \mathscr{L}_{f}\varphi$ defined by
\[
\mathscr{L}_{f}\varphi(x)
\equiv
\int_{E} e^{f(ax)}\varphi(ax)\, dp(a),
\quad \text{where}\ ax\equiv (a,x_1,x_2,\ldots).
\]
In particular, it follows by induction that, for all $n \in \mathbb{N}$, 
$dp^n (a_1,\ldots, a_n)$\break $\equiv  dp(a_1) \cdots dp(a_n)$ and 
$f_n(x) \equiv  \sum_{k=0}^{n-1} f(\sigma^k(x))$,
\begin{align*}
\mathscr{L}^n_f(\varphi)(x)
&=
\int_{E^n} e^{f_n(a_1\ldots a_n x)}
\varphi(a_1\ldots a_nx)\, dp(a_1) \cdots dp(a_n) \\
&\equiv 
\int_{E^n} e^{f_n(a x)}\varphi(ax)\, dp^n(a).
\end{align*}

Since $\|\mathscr{L}_{f}1\|_{\infty}<+\infty$ the Ruelle operator 
is bounded and the action of its dual (or Banach transpose) $\mathscr{L}_{f}^{*}$ 
on a generic element $\mu\in rba(X)$ is determined by
\[
\int_{X} \varphi\, d[\mathscr{L}_{f}^{*}\mu] 
= 
\int_{X} \mathscr{L}_{f}(\varphi)\, d\mu,
\quad \forall \varphi\in C_{b}(X,\mathbb{R}). 
\]

\begin{remark}
	Assume $E=\{1,\ldots,d\}$, the a priori measure $p$ is the normalized counting measure on $E$, 
	and $f$ is a continuous potential. Then we have, for all $\varphi\in C(X,\mathbb{R})$
	\[
	\mathscr{L}_{f}(\varphi)(x)
	=
	\int_{E} e^{f(ax)}\varphi(ax)\, dp(a)
	=
	\sum_{y\in \sigma^{-1}(x)} e^{\widetilde{f}(y)}\varphi(y),
	\]
	where $\widetilde{f}\equiv f-\log d$. This shows thus that, 
	in this particular setting, the Ruelle operator
	associated to a potential $f$ considered here coincides with the classical Ruelle operator
	but associated to a potential that differs from the original one by a constant.
\end{remark}

In order to motivate the concepts of pressure and entropy introduced 
in Section \ref{sec-pressure-entropy}, we prove in the sequel a 
Perron-Frobenius-Ruelle theorem for bounded H\"older potentials.


\section{Perron-Frobenius-Ruelle Theorem}\label{sec-RPF}
In this section we are interested in the space of 
bounded H\"older continuous functions 
$
\mathrm{Hol}(\alpha)
\equiv 
\mathrm{Hol}_{\alpha}(X,\mathbb{R})
,$   for  $0< \alpha< 1$, which
is defined as the space 
$\{f \in C_b(X,\mathbb{R}): \mathrm{D}_\alpha(f)<\infty \}$, where
\[ 
\mathrm{D}_{\alpha}(f) 
\equiv  
\sup_{x\neq y}
\dfrac{|f(x)-f(y)|}{d_{X}(x,y)^{\alpha}}.
\]

Combining H\"older continuity of $f$ 
with $d(\sigma^{n}(x),\sigma^{n}(y)) = 2^{n}d(x,y)$, which is valid 
for points having the same first $n$ coordinates, 
it follows from a standard argument that there exists $C_f$ such that 
\begin{equation}\label{eq:GM}
\left|1-e^{f_n(a y) -f_n(a x)}\right| 
\leq 
C_f  \, d(x,y)^\alpha  
\quad \forall x,y \in X \hbox{ and } a \in E^n.
\end{equation}

By \eqref{eq:GM}, it is now easy to see that 
$\mathscr{L}^n_{f}$ maps $\mathrm{Hol}(\alpha)$ to itself. 
Namely, for $f,\varphi \in \mathrm{Hol}(\alpha)$, 
and $x,y\in X$, we have 
\begin{eqnarray*}
|\mathscr{L}^n_f(\varphi)(x) - 
\mathscr{L}^n_f(\varphi)(y)|
&\leq
\int_{E^n} 
e^{f_n(a x)} 
\left| \left(1-e^{f_n(a y) -f_n(a x)}\right) \varphi(ax)\right| dp^n(a) 
\\
&\qquad + \int_{E^n} e^{f_n(a y)} \left| \varphi(ax) - \varphi(ay) \right| dp^n(a)
\\[0.2cm]
&\leq 
\left(
C_f \|\mathscr{L}^n_f(\varphi)\|_{\infty} 
+ 2^{-n} \mathrm{D}_\alpha(\varphi)\|\mathscr{L}^n_f1\|_{\infty} 
\right) 
d(x,y)^\alpha.
\end{eqnarray*}

Instead of constructing an $\mathscr{L}_f$-invariant function through application 
of the Arzel\`a-Ascoli theorem and then normalizing $\mathscr{L}_f$, 
we consider the family of operators $\{\mathbb{P}^m_{n}\}$ defined by, 
for $m\in \mathbb N$ and $n \in \mathbb N \cup \{0\}$, 
\[ 
\mathbb{P}^m_{n} (\varphi) 
\equiv 
\frac{\mathscr{L}^{m}_f( \varphi \mathscr{L}^{n}_f(1))}
{\mathscr{L}^{m+n}_f(1) }.
\]
Observe that, by construction, $\mathbb{P}^m_{n} (1) = 1$ and  
$\mathbb{P}^m_{k+l} \circ \mathbb{P}^k_{l} = \mathbb{P}^{k+m}_{l}$. 
Furthermore, the proof of Lemma 2.1 in \cite{MR3568728} 
is also applicable to the situation in here and gives that 
\[ 
\mathrm{D}_\alpha(\mathbb{P}^m_n(\varphi))
\leq  
C_f (2 \|\varphi\|_\infty + 2^{-m}\mathrm{D}_\alpha(\varphi)).
\]
As shown in \cite{MR3568728,StadlbauerZhang:2017a}, this estimate and the fact that $X$ is a full shift allows to deduce the following. With respect to the equivalent metric 
\[ \overline{d}(x,y)\equiv  \max\{ 1, 4 C_f d_X(x,y)^\alpha \}, \]
the space $(X,\overline{d})$ is separable and complete. In particular, as the diameter of $(X,\overline{d})$ is finite, 
the space $\mathscr{M}_{1}(X)$ is separable and 
complete with respect to the Wasserstein metric $d$  (\cite{MR2401600,MR2483740,MR1105086}), which is equal to, through Kantorovich's duality, 
\begin{equation} 
\label{eq:Wasserstein-through-Kantotrovich} 
d(\mu,\nu) \equiv 
\sup\left\{ 
		\int f d(\mu -\nu) :  f \hbox{ with } \sup\frac{|f(x)-f(y)|}{\overline{d}(x,y)} \leq 1 
 	\right\}. 
\end{equation}
The action of the operators $\{\mathbb{P}^m_{n}\}$  on the space of $\overline{d}$-Lipschitz functions then allows to deduce, following in verbatim the proof of Theorem 2.1 in \cite{MR3568728}, 
that their dual action on the space of probability measures strictly contracts the Wasserstein metric for some $m \in \mathbb{N}$ and uniformly in $n$.

Since $\mathbb{P}^m_{n}$ contracts $d$, it immediately follows from 
the composition rule that, for any probability measure $\nu_0\in \mathscr{M}_{1}(X)$, 
the sequence $((\mathbb{P}^m_0)^\ast(\nu_0))_{m\in\mathbb{N}}$ 
is a Cauchy sequence and therefore converges to a 
probability measure $\nu$, which, again by contraction, is independent 
from $\nu_0$. It then  follows  as in \cite{MR3568728} that $\nu$ 
is a conformal measure, 
that is, $\mathscr{L}^{\ast}_f(\nu) = \lambda \nu$ for some $\lambda >0$. Observe that, by conformality, $\lambda = \int \mathscr{L}(1)d\nu$.    
Moreover, by contraction, $\nu$ is the unique 
measure with this property (\cite[Prop. 2.1]{MR3568728}).
Furthermore, it follows from \eqref{eq:GM} that 
\[
\sup_{x,y,n} \frac{\mathscr{L}_{f}^{n}(1)(x)}{\mathscr{L}_{f}^{n}(1)(y)} < \infty,
\] 
which  implies, again by the contraction property, that, 
with $\delta_x$ referring to the Dirac measure in $x$,
\begin{eqnarray}\label{eq-auto-funcao} 
h(x) 
\equiv \lim_{n \to \infty} 
\frac{\int_{X} \mathscr{L}_{f}^{n}(1)d\delta_x}{\int_{X} \mathscr{L}_{f}^{n}(1)d\nu } 
=   
\lim_{n \to \infty} 
\frac{\mathscr{L}_{f}^{n}(1)(x)}{\int_{X} \mathscr{L}_{f}^{n}(1)d\nu } 
\end{eqnarray} 
exists for each $x \in X$ and is bounded away from $0$ and $\infty$. 
Similar to $\nu$, it follows that $\mathscr{L}_f(h) = \lambda h$ 
and that, up to a multiplication by a scalar, $h$ is the unique H\"older function with 
this property (\cite[Prop. 2.2]{MR3568728}). 
Moreover, the following version of exponential 
decay holds (\cite[Th. A]{MR3568728}).

\begin{theorem}\label{teo-hol-const-corr}
 There exist $C > 0$ and $s \in (0,1)$ such that, for $\varphi,\psi \in \mathrm{Hol}(\alpha)$
and $\psi>0$, 
\[ 
D_{\alpha}
\left(
	\frac{\mathscr{L}^{n}(\varphi) }{\mathscr{L}^{n}(\psi)}  
	- 
	\frac{\nu( \varphi) }{\nu( \psi)} 
\right) 
\leq 
C s^n 
\left(
	\mathrm{D}_\alpha(\varphi) 
	+\left| \frac{\nu(\varphi)}{\nu(\psi)} \right|
\mathrm{D}_\alpha(\psi) 
\right) \|1/\psi\|_\infty. 
\]
\end{theorem}

We remark that Theorem \ref{teo-hol-const-corr} applied to $\psi= h$ (normalized 
eigenfunction in the sense that $\nu(h)=1$) and $\varphi=1$ give
the following estimates    
$D_{\alpha}( \mathscr{L}^{n}_{f}(1)/(\lambda^{n}h) -1)
\leq 
2CD_{\alpha}(h) \|h\|_\infty^{-1} s^n.$
Therefore we have, uniformly in $x\in X$
\begin{eqnarray}\label{eq-est-Ln-Lambdan}
1-2C s^n
\leq
\frac{\mathscr{L}^{n}_{f}(1)(x) }{\lambda^nh(x)}  
\leq 
1+2C s^n.
\end{eqnarray}
Since $0<s<1$ and $\|\log h\|_\infty < \infty$, it follows that 
$n^{-1}\log \mathscr{L}^{n}_{f}(1)(x)
\to \log \lambda.$
Furthermore, $\lambda=\rho(\mathscr{L}_{f}|_{ \mathrm{Hol}(\alpha) })$ the spectral radius 
of the action of $\mathscr{L}_{f}$ on $\mathrm{Hol}(\alpha)$.

\bigskip 
 
As an another application of the above theorem, one obtains almost immediately 
quasi-compactness of the normalized operator. 
In order to define the relevant operators and norms, 
let $h$ refer to the function as constructed above and, 
for $\varphi:X \to \mathbb{R}$ bounded and measurable, 
set $\|\varphi\|_\alpha \equiv \| \varphi\|_\infty + D_\alpha(\varphi)$ 
and
\[ 
Q(\varphi)(x) 
\equiv 
\frac{\mathscr{L}_{f}(h \varphi)(x)}{\lambda h(x)}, \quad \Pi(\varphi)(x) 
\equiv \int_{X} \varphi h\,  d\nu.  
\]
\begin{proposition}\label{prop:quasicompactness}
$\Pi$ and $Q$ act on $\mathrm{Hol}(\alpha)$ as bounded operators, and $\Pi Q = Q \Pi =\Pi$. Furthermore,   
$\|(Q - \Pi)^n\|_\alpha \leq Cs^n$, where $s$ is as in 
Theorem \ref{teo-hol-const-corr}, and the splitting 
$\mathrm{Hol}(\alpha) \equiv \mathbb{R} \oplus \ker(\Pi)$ into closed subspaces, 
with $\mathbb{R}$ standing for the constant functions, is invariant under $Q$ and $\Pi$. 
Furthermore, $Q|_\mathbb{R} = \Pi|_\mathbb{R} = \mathit{id}$.
\end{proposition}

\begin{proof}
Observe that $h \in \mathrm{Hol}(\alpha)$ is bounded from above and below. 
Hence, $\varphi h \in \mathrm{Hol}(\alpha)$ 
and $\varphi/h \in \mathrm{Hol}(\alpha)$ for any 
$\varphi \in \mathrm{Hol}(\alpha)$ which implies that $Q$ 
acts on $\mathrm{Hol}(\alpha)$. 
Furthermore, using conformality of $\nu$ and invariance of $h$, 
\begin{eqnarray*}
\Pi \circ Q (\varphi) 
&=&  
\int_{X} \lambda^{-1} {\mathscr{L}_{f}(h \varphi)}\, d\nu 
=   
\int_{X} h \varphi\, d \nu  
=  
\Pi(\varphi)
\\[0.3cm] 
Q \circ\Pi(\varphi)  
&=& 
\frac{\mathscr{L}_{f}(h \Pi(\varphi))}{\lambda h} 
= \Pi(\varphi).
\end{eqnarray*}
Hence, $\Pi Q = Q \Pi =\Pi$, and, in particular, $(Q - \Pi)^n = Q^n - \Pi$. 
Hence, by Theorem \ref{teo-hol-const-corr} applied to 
$h \varphi$ in the numerator and $ h$ in the denominator,
\begin{eqnarray*}
D_\alpha( (Q - \Pi)^n(\varphi)) 
&=& D_\alpha( Q^n(\varphi) - \Pi(\varphi))  
\\
&=& 
 D_\alpha\left( Q^n\left(\varphi - \Pi(\varphi)\right)\right) 
\\
&\leq&
C s^n \left(D_\alpha (h\varphi) + |\Pi(\varphi - \Pi(\varphi))| 
		\mathrm{D}_\alpha(h)  \right) \|1/h\|_\infty 
\\
& \leq& 
C \|1/h\|_\infty   s^n \left( \|h\|_\infty D_\alpha (\varphi) 
   +\|\varphi\|_\infty \mathrm{D}_\alpha(h)  \right) 
\\
&\leq& 
C^\ast s^n\| \varphi \|_\alpha.
\end{eqnarray*}
As $\int_{X} (Q - \Pi)^n(\varphi) h\, d\nu = 0$, 
it follows from $\sup_{x,y} d_X(x,y)=1$ 
that 
\[
\|(Q - \Pi)^n(\varphi)\|_\infty \leq  D_\alpha( (Q - \Pi)^n(\varphi)).
\]
Hence, $\|(Q - \Pi)^n(\varphi)\|_\alpha \leq C^\ast   s^n\| \varphi \|_\alpha$. 
The remaining assertion is obvious. 
\end{proof}

Provided that $\mathscr{L}_{f}(1)=1$, the above splitting now allows to apply the  
very general version of Nagaev's method by Hennion and Herv\'e in \cite{MR1862393} as follows. 
As  the space of complex-valued Hölder continuous functions 
$\mathfrak{B}$ is a Banach algebra,  
 condition $\mathcal{H}[1]$ of Hennion and Herv\'e is satisfied. Furthermore, condition $\mathcal{H}[2]$ in there follows from Proposition \ref{prop:quasicompactness}. 
 Now assume that $\xi$ is a real-valued Hölder continuous function and that $t \in \mathbb{R}$. By Lemma VIII.10 in \cite{MR1862393}, the  operator $\mathscr{L}_{f + i\xi t}$ acts as bounded
operator on $\mathfrak{B}$ and is analytic in $t$. Hence, also $\mathcal{H}[3]$ and $\hat{\mathcal{D}}$ are satisfied and Theorems A, B and C in \cite{MR1862393} are applicable. 

In order to state the result, set $S_n(\xi) \equiv  \sum_{k=0}^{n-1}\xi\circ \sigma^k$ and
 recall that $\xi$ is referred to as a non-arithmetic observable if the spectral radius of 
$\mathscr{L}_{f + i\xi t}$ is smaller than 1 for each $t \neq 0$.

\begin{proposition} Assume that $\xi$ is a real valued Hölder continuous function such that $\int \xi d\nu =0$.
Then $s^2 = \lim_n \frac{1}{n} \int (S_n(\xi))^2 d\nu $ exists and the following versions of central limit theorems (CLTs) hold. In there, $Z$ refers to a $N(0,s)$-distributed random variable.
\begin{enumerate}
\item (CLT with rate). If $s > 0$, then there exists $C> 0$ such that 
\[ \sup_{u \in \mathbb{R}}\left| \nu\left(\{x \in X: S_n(\xi)(x) \leq u \sqrt{n} \} \right) - P(Z \leq u)\right| \leq Cn^{-\frac{1}{2}}. \]
\item (Local CLT). If $s > 0$ and $\xi$ is non-arithmetic, then for any $g:\mathbb{R}  \to \mathbb{R}$ continuous with $\lim_{|u| \to \infty} u^2g(u) =0$,
\[ \lim_{n \to \infty}  \sup_{t \in \mathbb{R}} 
\left| \sqrt{2\pi n} s \int g (S_n(\xi) - u)d \nu  - e^{-\frac{u^2}{2ns^2}} E(g(Z)) \right| =0  .\]
\end{enumerate}
\end{proposition}

\section{Pressure, entropy and their equilibrium states}\label{sec-pressure-entropy}

In this section we define the concepts of entropy and pressure considered here. 
Before proceeding we recall that in the context of
uncountable alphabets, both entropy and pressure are usually 
introduced as $p$-dependent concepts, 
see for example \cite{MR2864625,MR1241537,MR2807681,MR3377291}. 
\bigskip 

We say that a potential $f\in C_{b}(X,\mathbb{R})$ is normalized if $\mathscr{L}_{f}1=1$.
Consider the set 
$
\mathscr{G}
\equiv 
\{
\mu\in \mathscr{M}_{1}(X): 
\mathscr{L}_{f}^{*}\mu=\mu \ \text{for some normalized potential}\ f\in \mathrm{Hol}(\alpha)
\}.
$
Following \cite{MR3377291}, we define the entropy of $\mu\in\mathscr{G}$ as
$
\mathrm{h}^{\mathrm{v}}(\mu) 
\equiv 
-\langle \mu, f\rangle
$,
where $f$ is some normalized potential in $\mathrm{Hol}(\alpha)$ arbitrarily chosen so that 
$\mathscr{L}_{f}^{*}\mu=\mu$. 
Actually, similarly to \cite{MR3377291} we can prove that for any $\mu\in \mathscr{G}$ we have
\begin{eqnarray}\label{def-entropy}
\mathrm{h}^{\mathrm{v}}(\mu) 
=
\inf_{g\in \mathrm{Hol}(\alpha)} -\langle \mu,g\rangle + \log \lambda_{g},
\end{eqnarray}
where $\lambda_{g}$ is the eigenvalue obtained in the last section. 

Since the above expression makes sense for any $\mu\in rba(X)$ we have a natural way to 
define the entropy of a bounded finitely additive measure.

Next we obtain a generalization of the classical variational principle. 
Before we will make a few observations 
and introduce some notations. 
We first observe that the constant function $f=1$ is in $C_{b}(X,\mathbb{R})$
and so the set of all finitely additive probability measures
\[
\mathscr{M}_{1}^{a}(X)
\equiv 
\bigcap_{ \substack{ f\in C_{b}(X,\mathbb{R}) \\ f\geq 0 } }
\{\mu\in rba(X): \mu(1)=1\ \text{and}\ \mu(f)\geq 0 \}
\]
is a closed subset of the closed unit ball 
$\{\mu\in rba(X):\|\mu\|_{TV}\leq 1 \}$
in the weak-$*$-topology. This fact together with the 
Banach-Alaoglu theorem implies that $\mathscr{M}_{1}^{a}(X)$ is a compact
space.

Note that the space of all non-negative shift-invariant finitely additive measures 
\[
\mathscr{M}_{\sigma}^{a}(X)
\equiv
\{\mu\in \mathscr{M}_{1}^{a}(X):\  \mu(f\circ\sigma)=\mu(f),\ \ \forall f\in C_{b}(X,\mathbb{R})  \}
\]
is also a compact, with respect to the weak-$*$-topology. 
Indeed, let $(\mu_{d})_{d\in D}$ a topological net in $\mathscr{M}_{\sigma}^{a}(X)$ and
suppose that $\mu_{d}\to \mu$, in the weak-$*$-topology. Then 
for any $g\in C_{b}(X,\mathbb{R})$ we have
$
\mu(g\circ \sigma)
=
\lim_{d\in D}  \mu_{d}(g\circ\sigma)
=
\lim_{d\in D}  \mu_{d}(g)
=
\mu(g),
$
where the last equality follows from the weak-$*$ continuity of $\mu\in rba(X)$.
Of course, $\mu(g)\geq 0$ whenever $g\geq 0$ and $\mu(1)=1$.

\begin{definition}(Pressure Functional)
\label{def-top-pressure}
The functional $P:C_{b}(X,\mathbb{R})\to \mathbb{R}$
given by 
\[
P(f) 
\equiv 
\sup_{\mu\in \mathscr{M}_{\sigma}^{a}(X) } 
\mathrm{h}^{\mathrm{v}}(\mu) +\langle\mu,f\rangle
\]
is called the pressure functional and the real number $P(f)$ is called
topological pressure of $f$. 
\end{definition}

\begin{proposition}
	The pressure functional $P$ is a 
	convex function on $C_{b}(X,\mathbb{R})$.
\end{proposition}
\begin{proof}
	The convexity follows immediately from Definition \ref{def-top-pressure}.
\end{proof}

Before we proceed we would like to explain why the theory, that will be developed below,
  is not comprised in \cite{MR753071}. In there, Phelps and Israel
developed an abstract theory of generalized pressure and 
presented some applications to lattice gases.
In their work, the space $X$ is supposed to be a metric compact space, and a 
pressure functional is any real-valued convex function $\mathscr{P}$ 
defined on $C_{b}(X,\mathbb{R})=C(X,\mathbb{R})$ 
satisfying the  conditions
\begin{enumerate}
	\item $\mathscr{P}(f+c)=\mathscr{P}(f)+c$,
	\item if $f\leq 0$, then $\mathscr{P}(f)\leq \mathscr{P}(0)$,
	\item if $f\geq 0$, then $\|q(f)\|\leq \mathscr{P}(f)$,
	\item if $g\in \mathscr{I}$, then $\mathscr{P}(f+g)=\mathscr{P}(f)$,
\end{enumerate}
where $c\in \mathbb{R}$ is a constant, $\mathscr{I}$ denotes 
the subspace of $C(X,\mathbb{R})$ generated by 
the set $\{g-g\circ\sigma: g\in C(X,\mathbb{R}) \}$ and 
$q:C(X,\mathbb{R})\to C(X,\mathbb{R})/\mathscr{I}$ is the quotient 
map. In \cite{MR753071}, when the authors introduce entropy,  condition
\textit{(3)} is replaced by a stronger one. This new condition, which we call \textit{(3')}, 
is a kind of coercivity condition. To be more precise,
it requires $\|f\|_{\infty}\leq \mathscr{P}(f)$ whenever $f\geq 0$. 
Afterwards, 
for a given 
pressure functional $\mathscr{P}$ satisfying \textit{(3')},
the authors define the entropy $\mathfrak{h}\equiv \mathfrak{h}(\mathscr{P})$ 
as  the Legendre-Fenchel transform 
of $\mathscr{P}$. Condition \textit{(3')} is then employed in \cite[Prop 2.2]{MR753071}   
to show that the entropy of any shift-invariant probability measure $\mu$ is bounded 
by $0\leq \mathfrak{h}(\mu)\leq P(0)$.
Although the results in \cite{MR753071} can be applied in several contexts, 
condition \textit{(3')} does not hold in general in Statistical Mechanics and 
Thermodynamic Formalism. For example, the specific entropy
considered in \cite{MR2807681} is not bounded from below. 
Actually, it is well-known in Statistical Mechanics that 
the ground state entropy can go to minus infinity for uncountable 
(even compact) spin spaces.

If $X=E^{\mathbb{N}}$,
where $E$ is an uncountable infinite compact metric space, 
the entropy considered in \cite{MR3377291} does neither satisfy \textit{(3)} nor \textit{(3')},
and the authors show that their entropy of a Dirac measure concentrated on a periodic 
orbit is not finite, see the remark to Proposition 5 in page 1939 in \cite{MR3377291}.
Note that the pressure functional introduced above in Definition \ref{def-top-pressure}
is another instance where condition \textit{(3')} might not hold. 
We also remark that in here, $X$ is not necessarily compact.

Our pressure functional, likewise in classical 
equilibrium Statistical Mechanics, depends on the Ruelle operator, 
which in turn depends on the \textit{a priori} measure $p$, so
the reader should keep in mind that our 
pressure functional is a $p$-dependent concept as well as will be
our concept of entropy. 
It is also worth noting that by taking a suitable 
\textit{a priori} measure, we recover the usual concept of topological pressure in 
a finite-alphabet setting.

\begin{definition}[Equilibrium States]
\label{def-eq-states}
Given a continuous potential $f\in C_{b}(X,\mathbb{R})$, we 
say that $\mu\in \mathscr{M}^{a}_{\sigma}(X)$
is a (generalized) equilibrium state for $f$ if 
\[
\mathrm{h}^{\mathrm{v}}(\mu) + \mu(f)
=
\sup_{\mu\in \mathscr{M}^{a}_{\sigma}(X)} 
\mathrm{h}^{\mathrm{v}}(\mu) +\langle\mu,f\rangle
\equiv 
P(f).
\]
The set of all equilibrium states for $f$  will be denoted
by $\mathrm{Eq}(f)$.
\end{definition}

\begin{theorem}\label{teo-exist-equi}
Given a continuous potential $f\in C_{b}(X,\mathbb{R})$ there 
is $\mu_{f}\in \mathscr{M}^{a}_{\sigma}(X)$
such that 
\[
\mathrm{h}^{\mathrm{v}}(\mu_{f}) + \langle\mu_{f},f\rangle
=
\sup_{\mu\in \mathscr{M}^{a}_{\sigma}(X)} 
\mathrm{h}^{\mathrm{v}}(\mu) +\langle\mu,f\rangle.
\]
\end{theorem}

\begin{proof}
From Definition \ref{def-entropy} follows that the mapping 
$
\mathscr{M}^{a}_{\sigma}(X)\ni \mu
\longmapsto 
\mathrm{h}^{\mathrm{v}}(\mu) + \mu(f)
$
is upper semi-continuous with respect to the weak-$*$-topology. 
Since $\mathscr{M}^{a}_{\sigma}(X)$ is compact and 
convex follows from the Bauer maximum principle 
that there exists some $\mu_{f}\in \mathscr{M}^{a}_{\sigma}(X)$ 
such that 
\[
\mathrm{h}^{\mathrm{v}}(\mu_{f}) +\mu_{f}(f)
=
\sup_{ \mu\in \mathscr{M}^{a}_{\sigma}(X) } 
\mathrm{h}^{\mathrm{v}}(\mu) +\mu(f).
\]
Moreover, the Bauer maximum principle ensures that we 
can take the finitely additive measure $\mu_{f}$, attaining the above supremum,
in such a way that $\mu_{f}$ is in the set of extreme points of $\mathscr{M}^{a}_{\sigma}(X)$.
\end{proof}

An equilibrium state $\mu_{f}$ as in
the previous theorem is not necessarily a countably additive measure.
On the other hand, the Yosida-Hewitt decomposition 
\cite[Theorem 1.23]{MR0045194} states that 
$\mu_{f}= (\mu_{f})_{c}+ (\mu_{f})_{a}$, where $(\mu_{f})_{c}$ is 
a non-negative countably additive measure and $(\mu_{f})_{a}$ is a
non-negative purely finitely additive measure. 
That is, if $\mu$ is a non-negative countably additive measure such 
that $\mu\leq (\mu_{f})_{a}$, then $\mu=0$.

At this point, we do not have complete information on how the regularity properties or
the shape of the graph of the potential are linked to this decomposition 
this seems to be a relevant and interesting problem.
On the other hand, we can prove other important properties about the set
$\mathrm{Eq}(f)$ consisting of all equilibrium states associated to a bounded 
continuous potential $f$. 

If $p=P|_{\mathrm{Hol}(\alpha)}$ then Theorem \ref{teo-exist-equi} ensures that the 
subdifferential 
\[
\partial p(f)
\equiv  
\{\mu \in rba(X): p(g)\geq p(f)+\langle\mu,g-f\rangle, \ \forall g\in \mathrm{Hol}(\alpha) \},
\]
at every $f\in \mathrm{Hol}(\alpha)$ is not empty and 
it is easy to see that $\mathrm{Eq}(f)= \partial p(f)$.
The next proposition  is  a trivial observation showing
that the restriction of $\mathrm{h}^{\mathrm{v}}$ to a subdifferential
$\partial p(f)$ at any $f\in \mathrm{Hol}(\alpha)$ is an affine function.

\begin{proposition}\label{prop-subdiff-and-affinity-entropy}
Let $f\in \mathrm{Hol}(\alpha)$ be a given potential and 
$\partial p(f)$ the subdifferential of $p$ at $f$.
Then the restriction $\mathrm{h}^{\mathrm{v}}|_{\partial p(f)}$ is 
an affine function. In particular, any $\mu\in \partial p(f)$ is an equilibrium
state for $f$.
\end{proposition}

\begin{proof} From the definition for any $\mu\in \partial p(f)$ we 
have $p(f)-\langle\mu,f\rangle \leq p(g)-\langle\mu,g\rangle$
for all $g\in \mathrm{Hol}(\alpha)$. Therefore
\[	
\partial p(f)\ni \mu\longmapsto
\mathrm{h}^{\mathrm{v}}(\mu)
=
\inf_{g\in \mathrm{Hol}(\alpha)}
p(g)-\langle\mu,g\rangle
=
p(f)-\langle\mu,f\rangle.
\qedhere
\]	
\end{proof}

\begin{proposition}
For any $f\in C_{b}(X,\mathbb{R})$ we have that $\mathrm{Eq}(f)$
is a compact and convex subspace of $\mathscr{M}_{\sigma}^{a}(X)$.
\end{proposition}

\begin{proof}
Let $\mu,\nu\in \mathrm{Eq}(f)$ and $\lambda\in [0,1]$. 
From elementary properties of the infimum and \eqref{def-entropy} 
follows that $\mathrm{h}^{\mathrm{v}}$ is concave function. Hence,
\begin{multline*}
\mathrm{h}^{\mathrm{v}}(\lambda\mu+(1-\lambda)\nu) 
+\langle\lambda\mu+(1-\lambda)\nu,f\rangle 
\\
\geq
\lambda\mathrm{h}^{\mathrm{v}}(\mu) 
+
(1-\lambda)\mathrm{h}^{\mathrm{v}}(\nu) 
+\lambda\langle\mu,f\rangle+(1-\lambda)\langle\nu,f\rangle
=
P(f),
\end{multline*}
thus proving that $\mathrm{Eq}(f)$ is a convex set. 
The compactness of $\mathrm{Eq}(f)$ follows from compactness 
of $\mathscr{M}_{\sigma}^{a}(X)$ and the upper semi-continuity
of $\mathrm{h}^{\mathrm{v}}$.
\end{proof}

\begin{remark}
It follows from the last proposition and the Krein-Milman theorem that
the set of extreme points of $\mathrm{Eq}(f)$, denoted by $\mathrm{ex}(\mathrm{Eq}(f))$,
is not-empty. In particular, it is natural to conjecture that any element 
in $\mathrm{ex}(\mathrm{Eq}(f))$ is an ergodic finitely additive measure. 
This usually is established by 
showing that  
$\mathrm{ex}(\mathrm{Eq}(f))=\mathrm{Eq}(f)\cap \mathrm{ex}(\mathscr{M}_{\sigma}^{a}(X))$ 
using that the 
the entropy is an affine continuous function
on $\mathscr{M}_{\sigma}^{a}(X)$.
However, this approach does not work in our setting for general \textit{a priori} 
measures and non-compact spaces
as $\mathrm{h}^{\mathrm{v}}$ restricted to $\mathscr{M}^{a}_{\sigma}(X)$
might no longer be affine. Actually, $\mathscr{M}^{a}_{\sigma}(X)$ 
contains infinitely many elements 
whose entropy is equal to minus infinity.  
Of course, in particular cases, e.g. if the potential is  
H\"older continuous, there are other techniques
 to establish the ergodicity of the extreme equilibrium states.
\end{remark}

\begin{remark}
Since $(C_{b}(X,\mathbb{R}) , rba(X))$ is a dual pair and $P$ 
is a proper convex function (that is, if its effective domain is nonempty
and $P$ never takes  the value $-\infty$) it follows from Corollary 7.17
in \cite{MR2378491} that $\mathrm{Eq}(f)$ is a  singleton if and only if $P$ is G\^ateaux 
differentiable at $f$. The differentiability of the pressure restricted to $\mathrm{Hol}(\alpha)$
was recently obtained  when $X$ is compact (see \cite{2017arXiv170709072S,MR3194082}) and is a classical
result for finite alphabets, see for example 
\cite{MR1793194,MR1085356,MR0234697,MR0390180,MR0466493}.
\end{remark}

On the other hand, if the potential is H\"older continuous, then the 
following result shows that $h d\nu$, with $h$ and   
$\nu$ as in Section \ref{sec-RPF}, is a countably 
additive equilibrium state.

\begin{theorem}\label{teo-equi-measure-f-holder}
	Let $f$ be a bounded H\"older potential. Then 
	there is at least one equilibrium state $\mu_{f}$, associated to $f$, such that 
	its Yosida-Hewitt decomposition has only the countably additive part.
	More precisely, this equilibrium state is given by $\mu_{f}=h\nu$, where $h$ is
	a suitable normalized eigenfunction associated to $\lambda_{f}$
	and $\nu$ is the eigenmeasure of the dual of the Ruelle operator.
\end{theorem}

\begin{proof}
Let $f$ be a H\"older potential. By the definition of entropy we have 
\begin{eqnarray*}
\sup_{ \mu\in \mathscr{M}^{a}_{\sigma}(X) } 
\mathrm{h}^{\mathrm{v}}(\mu) +\langle \mu,f\rangle 
&=&
\sup_{ \mu\in \mathscr{M}^{a}_{\sigma}(X) } 
\big[ \inf_{g\in \mathrm{Hol}(\alpha)} -\langle \mu,g\rangle + \log \lambda_{g}\big] +\langle \mu,f\rangle
\\
&\leq&
\sup_{ \mu\in \mathscr{M}^{a}_{\sigma}(X) } 
-\langle \mu,f\rangle + \log \lambda_{f} +\langle \mu,f\rangle
\\
&=&
\log \lambda_{f}.
\end{eqnarray*}

Since $f$ is a H\"older potential, we can use the Perron-Frobenius-Ruelle Theorem of Section 
\ref{sec-RPF} to find a normalized potential $\bar{f}\in \mathrm{Hol}(\alpha)$ cohomologous to $f$, that is, 
$\bar{f}=f+\log h-\log h\circ\sigma -\log \lambda_{f}$. 
It is easy to see that $h$, up to a positive constant, 
can be chosen so that $\mu_{f}\equiv h\nu \in \mathscr{M}_{1}(X)$
and $\mathscr{L}_{\bar{f}}^{*}\mu_{f}=\mu_{f}$. Therefore 
$\mu\in\mathscr{G}\cap \mathscr{M}^{a}_{\sigma}(X)$ and by definition we have
$\mathrm{h}^{\mathrm{v}}(\mu_f) = -\langle \mu_{f}, \bar{f}\rangle = -\langle \mu_{f}, f\rangle +\log\lambda_{f}$.
This equality implies 
\[
\log \lambda_{f} = \mathrm{h}^{\mathrm{v}}(\mu_f) + \langle \mu, f\rangle
\leq 
\sup_{ \mu\in \mathscr{M}^{a}_{\sigma}(X) } 
\mathrm{h}^{\mathrm{v}}(\mu) +\langle \mu,f\rangle
\]
which together with the last inequality ensures that $\mu_{f}$ is an equilibrium state.
\end{proof}

\section{Extreme Positive $rba(X)$ measures in the\break Closed Unit Ball are uniquely Maximizing}\label{sec:extreme_positive}

The aim of this section is to obtain a result similar to 
the main result of \cite{MR2226487} in a non-compact setting. 
The techniques developed in \cite{MR2226487} are not applicable here mainly 
because $C_{b}(X,\mathbb{R})$ may not be separable and the induced 
weak-$*$-topology on the closed unit ball of its dual is not necessarily metrizable.

\begin{theorem}\label{Functional}
 Let $S$ be an arbitrary topological space,  $C_b(S,\mathbb{R})$ denote 
 the Banach space of all real-valued bounded continuous functions on $S$ 
 endowed with the supremum norm, and $\mathscr{M}^{a}_{1}(S)$ the subset of the 
 topological dual $C_b(S,\mathbb{R})^{*}$, consisting of those functionals 
 which have norm one and are  mapping the positive cone 
 $C_b(S,\mathbb{R})_+$ into $[0,\infty)$. 
 For each $\mu \in \mathscr{M}^{a}_{1}(S)$, the following assertions are equivalent.
 \begin{itemize}
\item[i)] $\mu$ is an extreme point of $\mathscr{M}^{a}_{1}(S)$, i.e. $\mu$ 
can not be written as a convex combination of two 
functionals in $\mathscr{M}^{a}_{1}(S) \setminus \{\mu\}$.

\item[ii)] $\mu$ is an \emph{exposed point} of $\mathscr{M}^{a}_{1}(S)$, that is, 
there exists a functional $\xi$ in the bi-dual $C_b(S,\mathbb{R})^{**}$ 
which attains its strict minimum on the set $\mathscr{M}^{a}_{1}(S)$ 
at the point $\mu$.

\item[iii)] There exists a functional $\xi$ in the bi-dual 
$C_b(S,\mathbb{R})^{**}$ which is zero at $\mu$ and strictly positive on $\mathscr{M}^{a}_{1}(S) \setminus \{\mu\}$.

\item[iv)] $\mu$ is a \emph{lattice homomorphism}, 
i.e. we have $|\langle \mu, f\rangle| = \langle \mu, |f| \rangle$ 
for all $f \in C_b(S, \mathbb{R})$.

\item[v)] $\mu$ is an algebra homomorphism, 
i.e. we have 
$\langle \mu, f_1f_2 \rangle = \langle \mu, f_1 \rangle \langle \mu, f_2 \rangle$ 
for all $f_1,f_2 \in C_b(S,\mathbb{R})$.
\end{itemize}

\end{theorem}

The proof of above theorem can be found in \cite{MO302479}.

\bigskip

For the next corollary we assume that $X=E^{\mathbb{N}}$, where
$(E,d_{E})$ is a non-compact standard Borel space satisfying the following 
property. There exists $a_0\in E$ and a sequence $(a_n)_{n\geq 1}$
of distinct points such that $d_{E}(a_0,a_{n-1})<d_{E}(a_0,a_n)$ and $d(a_0,a_n)\to\mathrm{diam}(E)$.
For the sake of simplicity, we also assume that $\mathrm{diam}(E)=1$ and $d(x,y)<1$, for all $x,y\in X$.

\begin{corollary}\label{cor:unique-additive-measure}
If $X$ is a non-compact space satisfying the above property, 
then there exists  an extreme, finitely additive measure in  
$\mathscr{M}_{\sigma}^{a}(X)\setminus\mathscr{M}_{\sigma}(X)$ (i.e., not necessarily countably additive measure) 
which  is the unique  maximizing measure  for some potential $f\in C_{b}(X,\mathbb{R})$.
\end{corollary}

\begin{proof}
For $n\geq 0$, let $x^{(n)}=(a_n,a_n,\ldots)\in X$ and consider 
the associated sequence of Dirac delta measures
$(\delta_{x^{(n)}})_{n\geq 1}$. By compactness of $\mathscr{M}^{a}_1(X)$, 
this sequence of measures, viewed as a topological net, has a convergent
subnet $(\delta_{ x^{(\alpha)} })_{\alpha\in D}$. 
Let $\mu=\lim_{d\in D}\delta_{ x^{(\alpha)} }$.
We claim that $\mu$ is not a countably additive measure.
Indeed, take $B_n = X\setminus  \{x\in X: d(x,x^{(0)})<d(x^{(0)},x^{(n)}) \}$.
Note that the hypothesis considered on $E$ imply $B_n\downarrow \emptyset$. 
Suppose by contradiction that $\mu$
is a countably additive measure. 
Since for each $n\geq 1$, the set $B_n$ is closed, 
follows from Portmanteau theorem (Theorem 6.1 item (c) of \cite{MR2169627})
\[
\mu(B_n) \geq \limsup_{\alpha\in D} \delta_{x^{(\alpha)}}(B_n)=1.
\]
Consequently, $\mu$ is not a countably additive measure which is a contradiction.

A straightforward computation shows that any such cluster point $\mu$
is a shift-invariant measure. It remains to  show that $\mu$ is an extreme point 
of $\mathscr{M}^{a}_1(X)$. This fact is a 
consequence of the equivalence $i)\Leftrightarrow v)$ 
of Theorem \ref{Functional}. Indeed, for each $\alpha\in D$ 
the measure $\delta_{ x^{(\alpha)} }$
is an extreme point of $\mathscr{M}^{a}_1(X)$ as 
$
\left<\delta_{ x^{(\alpha)} }, f_1f_2\right>
=
\left<\delta_{ x^{(\alpha)} }, f_1\right>
\left<\delta_{ x^{(\alpha)} }, f_2\right>
$. 
In order to conclude that $\mu$ satisfies a similar relation 
it is enough to observe that the above equality is stable under weak-$*$ limits 
so we have 
$\left<\mu, f_1f_2\right>=\left<\mu, f_1\right>\left<\mu, f_2\right>$. 
By using again the equivalence $i)\Leftrightarrow v)$ of Theorem \ref{Functional}
it follows that $\mu$ is an extreme point.

Let $\xi:C_b(X,\mathbb{R})^{**}\to\mathbb{R}$ be the linear functional
obtained in item iii) of Theorem \ref{Functional} 
to $\mu$. Recall that $\xi$ is of the form $\xi(\nu)=\left<\nu, g\right>$ 
for some $g$ in $C_b(X,\mathbb{R})$, see \cite[Proposition 3.14]{MR2759829}.
Finally, by taking the potential $f=-g$
and considering the functional 
$F\in C_b(X,\mathbb{R})^{*}$ defined by $F(\mu)=\left<\mu, f\right>$ 
the result follows. 
\end{proof}

\section{Applications}\label{sec:applications}

\subsection{Finite Entropy Ground-States and Maximizing Measures}\label{subsec:Finite Entropy Ground-States and Maximizing Measures}

In this section we consider the following ergodic optimization problem.
We fix a potential $f\in C_{b}(X,\mathbb{R})$ and consider the problem 
of finding an element of $\mathscr{M}_\sigma^a(X)$ with finite entropy 
which attains the  supremum 
\[
m(f)
=
\sup_{\substack {\nu \in \mathscr{M}_\sigma^a(X)\\[0.05cm] \mathtt{h}^{\mathtt{v}}(\nu)>-\infty}}
\ \int_{X} f d\nu.
\] 
An invariant measure $\mu$ having finite entropy is 
referred to as a \textit{ maximizing measure} for 
the potential  $f$  if it attains the supremum in the above 
variational problem, that is,
\[
m(f)
=
\sup_{\substack {\nu \in \mathscr{M}_\sigma^a(X)\\ \mathtt{h}^{\mathtt{v}}(\nu)>-\infty}}
\int f d\nu
=
\int f d\mu.
\] 
The above supremum is always finite since $f\in C_{b}(X,\mathbb{R})$ 
but the existence of a maximizing measure is a non-trivial problem
because the subset of functionals in $\mathscr{M}_\sigma^a(X)$ with finite entropy
is non-compact. 

Consider a fixed bounded H\"older potential $f$ and a real parameter $\beta>0$. 
We denote by $\mu_{\beta f}$ the equilibrium state constructed above 
associated to the potential $\beta f$.
We now show that any cluster point $\mu_{\infty}$
of the family $(\mu_{\beta f})_{\beta>0}$, such that 
$\mathtt{h}^{\mathtt{v}}(\mu_{\infty})>-\infty$ 
is a maximizing measure for $f$.   
It is standard to call $\mu_{\infty}$  a 
\emph{Gibbs State at zero temperature for the potential} $f$ or simply a
\emph{ground state} for $f$.

\begin{theorem}
	Let $f$ be a bounded continuous potential, $\beta>0$ 
	and $\mu_{\beta f}\in \mathrm{Eq}(\beta f)$. Suppose there 
	is at least one cluster point $\mu_{\infty}$ of $(\mu_{\beta f})_{\beta>0}$ 
	having finite entropy. Then $\mu_{\infty}$ 
	is a maximizing measure for the potential $f$.
\end{theorem}

\begin{proof}
Let $\mu_{\infty}$ be an arbitrary cluster point of 
the family $(\mu_{\beta f})_{\beta> 0}$, 
such that $\mathtt{h}^{\mathtt{v}}(\mu_{\infty})>-\infty$.
Note that for all $\beta>0$ we have that 
$\mathtt{h}^{\mathtt{v}}(\mu_{\beta f})>-\infty$ 
and $\mu_{\beta f}\in \mathscr{M}_\sigma^a(X)$. 
Therefore,
\begin{equation*}
\int_{X} f \, d\mu_{\infty} 
= 
\lim_{\beta\to  \infty }\int_{X} f\, d\mu_{\beta f}
\leq 
m(f).
\end{equation*}

On the other hand, for any $\nu\in \mathscr{M}_\sigma^a(X)$,  
we get from the variational principle that
$
\langle \beta f, \mu_{\beta f}\rangle + \mathtt{h}^{\mathtt{v}}(\mu_{\beta f})\geq 
\langle \beta f, \nu\rangle +\mathtt{h}^{\mathtt{v}}(\nu),
$
and that the inequality is non-trivial if $\mathtt{h}^{\mathtt{v}}(\nu)>-\infty$. 
In this case, 
\[
\int_{X} f\, d\mu_{\beta f}+\dfrac{\mathtt{h}^{\mathtt{v}}(\mu_{\beta f})}{\beta}
\geq 
\int_{X} f \,d\nu+\dfrac{\mathtt{h}^{\mathtt{v}}(\nu)}{\beta}
\]
and consequently, by the non positivity of $\mathtt{h}^{\mathtt{v}}$, 
\begin{equation*}
\int_{X} f\, d\mu_{\infty}=\lim_{\beta \to \infty}\int f d\, \mu_{\beta f}
\geq
\lim_{\beta \to \infty} 
\int_{X} f\, d\nu+\dfrac{\mathtt{h}^{\mathtt{v}}(\nu)}{\beta}
=
\int_{X} f\, d\nu.
\end{equation*}
Since this inequality holds for any $\nu\in \mathscr{M}_\sigma^a(X)$ 
having finite entropy, the result follows.
\end{proof}

\begin{remark}
	We remark that it is not possible to conclude from the previous proof that 
	the cluster point $\mu_{\infty}$ considered above is a countably additive measure. 
	If we do not require that $\mathtt{h}^{\mathtt{v}}(\mu_{\infty})>-\infty$, then 
	the above argument still gives us the inequality 
		\[
	\sup_{\substack {\nu \in \mathscr{M}_\sigma^a(X)\\ \mathtt{h}^{\mathtt{v}}(\nu)>-\infty}}
	\int f d\nu
	\leq 
	\int f d\mu_{\infty},
	\] 
	which, in principle,  could be strict.
\end{remark}

\subsection{Markov Chains on Standard Borel Spaces}
\label{sec-applications3}
In this section we show how to apply the results obtained here to
discrete time Markov Chains taking values in a metric space $E$. 
We then show how to construct and prove some stability results 
in \cite{MR2509253} within the framework of Thermodynamic Formalism.

Roughly speaking, a discrete-time Markov chain $\Phi$ on a metric space $E$ is a countable
collection $\Phi\equiv \{\Phi_0,\Phi_1,\ldots\}$ of random variables,
with $\Phi_i$ taking values in $E$ so that its future trajectories 
depend on its present and its past only through the current value. 
A concrete construction of a discrete time Markov chain, can be made by specifying a measurable 
space $(X,\mathscr{F})$, where each element of $\Phi$ is defined, an initial probability distribution 
$p:\mathscr{B}(E)\to [0,1]$, and a transition probability kernel $P:E\times \mathscr{B}(E)\to [0,1]$
such that
\begin{itemize}
	\item[i)] for each fixed $A\in \mathscr{B}(E)$ the map $a\longmapsto P(a,A)$
	is a $\mathscr{B}(E)$-measurable function,
	\item[ii)] for each fixed $a\in E$ the map $A\longmapsto P(a,A)$ is a 
	Borel probability measure on $E$.
\end{itemize}
\begin{definition}\label{def-markov-chain}
	A stochastic process $\Phi$ defined on $(X,\mathscr{F},\mathbb{P}_{\mu})=(E^{\mathbb{N}},\mathscr{B}(E^{\mathbb{N}}),\mathbb{P}_{\mu})$ 
	and taking values on $E$ is called a time-homogeneous Markov Chain, with transition probability kernel 
	$P$ and initial distribution $\mu$ if its finite dimensional distributions satisfy, for each $n\geq 1$,
	\begin{multline*}
	\mathbb{P}_{\mu}(\Phi_0\in A_{0},\ldots,\Phi_n\in A_{n})
	\\=
	\int_{A_0}\ldots\int_{A_{n-1}}
	P(y_{n-1},A_n)dP(y_{n-2},y_{n-1})\ldots dP(y_0,y_1)d\mu(y_0).
	\end{multline*} 
\end{definition}

\begin{definition}[Invariant Measures]
	A sigma-finite measure $\pi$ on $\mathscr{B}(E)$ with the property
	\[
	\pi(A) = \int_{E} P(x,A)\, d\pi(x)
	\]
	will be called invariant.  
\end{definition}

The key results about the existence of invariant measures for a Markov chain 
are based on recurrence, see for example 
Theorem 10.0.1 in \cite{MR2509253}. In what follows, we prove the existence
of such measures for a certain class of kernels based on the results of Section \ref{sec-RPF}.  
In order to do so, assume that  $f\in \mathrm{Hol}(\alpha)$ is a summable potential with 
respect to some \textit{a priori} measure $p$ on $E$, that is 
$\|\mathscr{L}_{f}(1)\|_\infty < \infty$. 
Then, for each $x=(x_1,x_2,\ldots)\in E^{\mathbb{N}}$, 
the map $A \mapsto \mathscr{L}_{f}(1_{A}\circ \pi_1)(x)$,  
for $x \in X$ and $A\in \mathscr{B}(E)$ defines a finite measure on $X$. 
In particular,  $dP(x,a) \equiv e^{f(ax)}dp(a)$, or equivalently,
\begin{align*}
P(x,A)
=
\int_{E} 1_{A}(a)\, dP(x,a)
\equiv
 \int_{E} e^{f(ax)} (1_{A}\circ \pi_1)(ax)\, dp(a)
=
\mathscr{L}_{f}(1_{A}\circ \pi_1)(x)
\end{align*}
defines a transition kernel 
which might be neither  a probability measure nor constant
 on $\{y \in E^{\mathbb{N}} : y_1 = x_1\}$. However, 
 it remains to check Kolmogorov's consistency conditions in order to 
 verify that $P$ defines a stochastic process. 
 That is, as $P$ induces the measure $\mathbb{P}_{x}$ on $E^{n}$ 
 with respect to the initial distribution $\delta_x$ for $x \in X$,  given by 
	\begin{multline*}
	\mathbb{P}_{x}(\Phi_1\in A_{1},\ldots\Phi_n\in A_{n}) =   \mathscr{L}_{f}\left(1_{\pi_1^{-1}A_1}  \mathscr{L}_{f}\left(1_{\pi_1^{-1}A_2} \cdots \mathscr{L}_{f}\left(1_{\pi_1^{-1}A_n} \right) \cdots \right)\right)(x) 
	\\
	=
	 \mathscr{L}^n_{f} \left(\textstyle \prod_{i=1}^n 1_{\pi_1^{-1}A_i}\circ \sigma^{n-i} \right)(x) 
	 = \mathscr{L}^n_{f} \left( 1_{\left\{(y_i) \in X:  y_i \in A_{n+1-i}, i=1,\ldots n \right\}} \right)(x),
	\end{multline*}  
it is necessary and sufficient that $ \mathscr{L}_{f}(1)=1$, 
or in other words, $f$ has to be normalized. Now let 
$\nu$ be the unique probability measure 
with $\mathscr{L}_{f}^{*}(\nu)=\nu$ as in Section \ref{sec-RPF}. 
With respect to this initial distribution, the above implies that 
	\begin{equation}\label{eq:natural_extension}
	\mathbb{P}_{\nu}(\Phi_1\in A_{1},\ldots,\Phi_n\in A_{n}) 
	= 
	\nu\left({\left\{(y_i) \in X:  y_i \in A_{n-i}, i=1,\ldots n \right\}} \right).
	\end{equation}
As the right-hand side is well known from the construction of the natural 
extension of a measure-preserving dynamical system, one obtains the following 
relation between $\sigma$ on $X$ and the stochastic process defined 
by $P$ with respect to a normalized potential $f$ through the 
bilateral shift on $(E^\mathbb{Z},\hat{\nu})$, where $\hat{\nu}$ is the extension 
of $\nu$ to $E^\mathbb{Z}$ through \eqref{eq:natural_extension}. 
In this setting, $\sigma$ corresponds to the left shift whereas  
$(\Phi_i:i\in \mathbb{N})$ is given by $\Phi_i = y_{-i}$ for $(y_i) \in E^\mathbb{Z}$. 
Furthermore, as $\mathscr{L}_{f}(1)=1$, it follows
from the same argument as in the proof of Theorem \ref{theo:assymptotically stable} that
\begin{equation}\label{eq:geoemtric_ergodicity}  
d\left(P^n(x,\cdot),\nu\right)
= 
d\left((\mathscr{L}^n_{f})^\ast(\delta_x),\nu\right)
\leq 
Cs^n d(\delta_x,\nu) 
\leq 
Cs^n,   
\end{equation}
where $d$ is the Wasserstein metric on the space of probability measures as defined in \eqref{eq:Wasserstein-through-Kantotrovich}
(see Theorem 1.1.5 in \cite{MR3058744}). 
Note that \eqref{eq:geoemtric_ergodicity} is also known as \emph{geometric ergodicity} 
in the literature on probability theory and that 
geometric ergodicity was established in 
\cite{MR3628927,MR3568728,StadlbauerZhang:2017a} 
for non-stationary and random countable shift spaces.

Observe that $(\Phi_i)$ in general is not a Markov chain as $P(x,\cdot)$ 
might not only depend on the the first coordinate of $x$. 
However, by assuming that $P(x,\cdot) = P(x_1,\cdot)$, 
or equivalently, that $f$ only depends on the first two coordinates, 
one easily obtains that $\pi\equiv \nu\circ \pi_1^{-1}$ is $P$-stationary as 
\begin{align*}
\int_{E} P(x_1,A)\, d\pi(x_1)
&=
\int_{E} P(x_1,A)\, d[\nu\circ\pi_1^{-1}](x_1)
\\
&=
\int_{E} \mathscr{L}_{f}(1_A\circ \pi_1)(x_1,x_2,\ldots)\, d[\nu\circ\pi_1^{-1}](x_1)
\\
&=
\int_{E^{\mathbb{N}}} \mathscr{L}_{f}(1_A\circ \pi_1)\, d\nu
=
\int_{E^{\mathbb{N}}} 1_A\circ \pi_1\, d[\mathscr{L}_{f}^{*}\nu]
=
\int_{E^{\mathbb{N}}} 1_A\circ \pi_1\, d\nu
\\
&=
\int_{E} 1_A\, d[\nu\circ\pi_1^{-1}] 
= 
\pi(A).
\end{align*}
Note that the argument depends on the assumption that $f$ only depends 
on the first coordinates, as if this would not be the case, 
the identity in line 3 would no longer be satisfied. 

\subsection{Asymptotic Stability of Markov Operators}
\label{sec-applications2}
In this section, we turn our attention to the closely related problem 
of asymptotic stability of Markov operators on standard Borel spaces and indicate
how some of the stability problems  
considered in \cite{MR1813364} can be approached by the results in Section \ref{sec-RPF}. 

Let $\mathscr{M}_{\mathrm{fin}}(X)$ be the set of all  finite nonnegative Borel measures on $X$.  
An operator $P:\mathscr{M}_{\mathrm{fin}}(X)\to \mathscr{M}_{\mathrm{fin}}(X)$ is 
called a \emph{Markov Operator} if it satisfies the following two conditions:
\begin{itemize}
	\item[(i)] \emph{positive linearity:}
	$
	P(\lambda_1\mu_1+ \lambda_2\mu_2)=\lambda_1P\mu_1+\lambda_2P\mu_2
	$, for all $\lambda_1, \lambda_2\geq 0$ and $\mu_1, \mu_2 \in \mathscr{M}_{\mathrm{fin}}(X)$,
	\item[(ii)] \emph{preservation of the norm:}
	$
	P\mu(X)=\mu(X)~~\text{for}~~\mu\in \mathscr{M}_{\mathrm{fin}}(X).
	$
\end{itemize}
A Markov operator is called a \emph{Feller operator} if there is a linear operator 
$U:C_b(X,\mathbb{R})\to C_b(X,\mathbb{R})$, the 
pre-dual to $P$, such that 
\[
\left<\mu,Uf\right>=\left<P\mu,f\right>~~\text{for}~~ 
f\in C_b(X,\mathbb{R}), \mu \in \mathscr{M}_{\mathrm{fin}}(X).
\]
Finally, a measure in $\mathscr{M}_{\mathrm{fin}}(X)$ is called 
\emph{stationary} if $P\mu=\mu$, and $P$ is called \emph{asymptotically stable} 
if there exists a stationary distribution $\nu$  such that 
\[
\lim_{n\to \infty}d(P^n\mu,\nu)=0, \quad \text{for all}\ \mu\in \mathscr{M}_1(X),
\]
where, as above, $d$ refers to the Wasserstein metric.
\begin{example}
	Let be $(E, \mathscr{E})$ a standard Borel space and $X=E^{\mathbb{N}}$ 
	the product space endowed with the product metric
	$
	d_X(x,y)=\sum_{n=1}^{\infty}1/2^n\min\{d_E(x_n,y_n), 1\}.
	$
	It is easy to see that $(X, d_X)$ is a Polish space.
	If $f$ is a bounded $\alpha$-H\"older continuous normalized potential, 
	then the restriction to $\mathscr{M}_{\mathrm{fin}}(X)$ of the 
	Banach transpose of the Ruelle operator $\mathscr{L}_f^{*}$ is a 
	Markov operator and its associated Feller operator is 
	$\mathscr{L}_f:C_b(X,\mathbb{R})\to C_{b}(X,\mathbb{R})$.
\end{example}

\begin{theorem}\label{theo:assymptotically stable}
	Under the assumptions of the above example, the Markov operator 
	$P=\mathscr{L}_f^{*}|_{\mathscr{M}_{\mathrm{fin}}(X)}$ is asymptotically stable. 
	Moreover, there exist $C>0$ and $s \in (0,1)$ such that, 
	where $\nu$ refers to the unique stationary probability measure and $d$ to the Wasserstein metric defined in \eqref{eq:Wasserstein-through-Kantotrovich}, 
	\[ d(P^n(\mu),\nu) \leq Cs^n \hbox{ for all } \mu \in \mathcal{M}_{1}.\]  
\end{theorem}

\begin{proof} As the potential is normalized, $\mathscr{L}_f(1) = 1$ and, in particular, $\mathbb{P}^m_{n} = \mathscr{L}_f^m$ and $(\mathbb{P}^m_{n})^\ast = P^m$ . Therefore, it follows from Theorem 2.1 in \cite{MR3568728} that there exists $t \in (0,1)$  and $m \in \mathbb{N}$ such that $d(\mathbb{P}^m_{n})^\ast(\mu), \mathbb{P}^m_{n})^\ast(\tilde{\mu}) \leq t d(\mu, \tilde{\mu})$. Moreover, as the $\overline{d}$-diameter of $X$ is 1, it follows 
that $d(\mu, \tilde{\mu})\leq 1$. Hence, for each $n \in \mathbb{N}$,  
\[  d(P^n(\mu),P^n(\tilde{\mu}) \leq t^{-1} t^{\frac{n}{m}}d(\mu, \tilde{\mu}).  \]
In particular, $P$ has a unique fixed point $\nu$ and, for $\tilde{\mu} \equiv  \nu$, it follows that  
$d(P^n(\mu),\nu) = d(P^n(\mu),P^n(\nu)) \leq t^{n/m -1}$. 
   \end{proof}

\subsection{Infinite Interacting Random Paths}
\label{sec-applications}

We consider the following random path process. 
At each discrete time $t=n\in\mathbb{N}$, a random point $q_n\in\mathbb{R}^d$ 
is chosen accordingly to the $d$-dimensional standard Gaussian measure
\[
G_{d}(A)=  
\frac{1}{(2\pi)^{n/2}}
\int_{A} \exp\left(-\frac{1}{2} \|v\|_{2}^{2}\right)
d\lambda^n(v).
\]
This sequence of random points induces a random path process on $\mathbb{R}^d\times [1,+\infty)$, 
given by the linear interpolation among these points, that is,
\begin{equation}\label{def-gamma-pn}
\gamma(t) = (1-(t-(n-1)))q_{n}+(t-(n-1))q_{n+1},
\quad \text{if} \ t\in [n,n+1].
\end{equation}
\begin{figure}[h]
\centering
\includegraphics[width=0.9\linewidth]{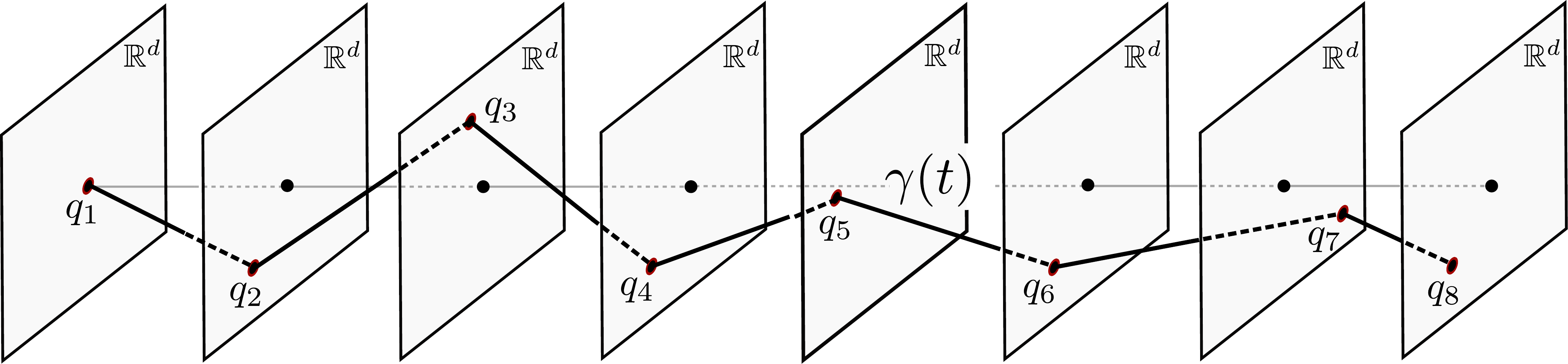}
\caption{An example of a random path $\gamma(t)$ constructed from $q_1,q_2,\ldots$}
\label{fig:caminhos}
\end{figure}

This construction induces a bijection $\Gamma:(\mathbb{R}^d)^{\mathbb{N}}\to \Upsilon$,
where $\Upsilon$ is the set of all ``polygonal'' paths of the form \eqref{def-gamma-pn}.

Let $p$ be the probability measure obtained by 
the pushforward of the infinite product measure 
$\prod_{i\in\mathbb{N}} G_{d}$ to $\Gamma$.
The space $\Upsilon$ of such 
all such paths has natural
structure of a standard Borel space inherited by $(\mathbb{R}^d)^{\mathbb{N}}$.
In the language of the  previous sections $E=\Upsilon$ and the \textit{a priori} 
measure $p$ is the push-forward of  $\prod_{i\in\mathbb{N}} G_{d}$.

Let $f:\Upsilon^{\mathbb{N}}\to\mathbb{R}$ be a H\"older bounded
potential. A point in $\Upsilon^{\mathbb{N}}$ will be denoted by $(\gamma_1,\gamma_2,\ldots)$.
Note that each coordinate $\gamma_n$ of a such 
point is actually a path in $\mathbb{R}^d\times [1,+\infty)$. 
Now we consider the Ruelle operator 
\[
\mathscr{L}_{f}(\varphi)(\gamma_1,\gamma_2,\ldots)
=
\int_{\Upsilon} \exp(f(\gamma,\gamma_1,\gamma_2,\ldots))
\ \varphi(\gamma,\gamma_1,\gamma_2,\ldots)\ dp(\gamma).
\]

Since we are assuming that $f$ is a bounded H\"older continuous 
Theorem \ref{teo-equi-measure-f-holder} implies the existence of 
an equilibrium measure $\mu_{f}$ which is also a countably additive Borel 
probability measure. This equilibrium measure $\mu_{f}$ describes what will be
the law of this infinite interacting random path process
in $\mathbb{R}^d\times [1,+\infty)$.
The interesting feature of this approach is to allow
the construction of an infinite interacting path process measure, 
having infinite-body interactions, since $f$ 
can be chosen as a function depending 
on infinitely many coordinates.

Interesting examples are obtained by the following 
class of potentials 
\[
f(\gamma_1,\gamma_2,\ldots) = 
-\sum_{n=1}^{\infty} J(n)
\frac{d_{\mathbb{H}}^{\alpha}(\gamma_1,\gamma_n)}
{1+d_{\mathbb{H}}^{\alpha}(\gamma_1,\gamma_n)}
\]
where $J(n)\geq 0$, and goes to zero sufficiently fast, 
$0<\alpha<1$ and $d_{\mathbb{H}}$ stands for the Hausdorff distance.
For each inverse temperature $\beta>0$ we consider the equilibrium measure 
$\mu_{\beta f}$. 

\begin{conjecture}
At very low temperatures $(\beta\gg 1)$ the typical 
configuration should be an infinite collection of paths which are closed to 
each other and also close to the origin (this last information comes from 
the dependence of $\mu_{f}$ on the \textit{a priori} measure $p$).
On the other hand, at very high temperatures $(0<\beta\ll 1)$ 
a typical configuration for $\mu_{\beta}$ should be 
similar to an infinite collection of independent ``diffusive'' paths. 
\end{conjecture}

The results of the previous section also allow us to construct 
a Markov process that can be used to describe the time evolution 
of this infinite interacting random path process.
Given a bounded H\"older potential $f$ we consider a normalized 
potential $\bar{f}$ cohomologous to $f$ and 
the following Markov pre-generator  
$T:C(\Upsilon^{\mathbb{N}},\mathbb{R})\to C(\Upsilon^{\mathbb{N}},\mathbb{R})$
given by 
\[
T:\mathscr{L}_{\bar{f}} - I.
\]
Clearly, this is actually a Markov generator since $\mathscr{L}_{\bar{f}}$ is 
bounded and everywhere defined operator. Therefore, we can 
apply the Hille-Yosida Theorem to construct a Markov semigroup 
$\{S(t):t\geq 0\}$ given by 
\[
S(t)(\varphi) = \lim_{n\to\infty} (I-(t/n)T)^{-n}
\varphi, \quad\forall \varphi\in C_{b}(\Upsilon^{\mathbb{N}},\mathbb{R})
\]
which is a diffusion in infinite dimension obtained 
from a potential which is not necessarily of finite-body type interaction.

Analogous considerations apply to the potential $\beta f$
so the semigroup associated to this potential 
should be ergodic as long as $J(n)$ decays to 
zero exponentially fast and $\beta$ is sufficiently small. 
Therefore for any choice of $\nu$ (countably additive probability measure), 
we have $S(t)^*(\nu)\rightharpoonup \mu_{\beta f}$.
This observation actually follows from the 
famous $(M-\varepsilon)$ theorem, see \cite{MR2108619}.

\begin{conjecture}
As long as the Ruelle operator has the spectral gap property
and the potential $f$ has continuous partial derivatives, intuitively, 
one would expect that 
the scaling limit (in the sense of Donsker theorem) of the 
infinite-dimensional Markov process associated to this semigroup 
is a formal solution of the
infinitely dimensional stochastic differential equation 
\[
dX^n_t = dB^n_t - \langle e_n, \nabla f(\sigma^{n}(X^1_t,X^2_t,\ldots))\rangle\, dt
\]
\end{conjecture}
This stochastic differential equation has its origin in the
works of Lang \cite{MR0431435,MR0455161}, where the potential $f$ 
has either one or two-body interactions, satisfies  
some symmetry and smoothness condition.
This equations are also studied using ideas of 
DLR-Gibbsian equilibrium states in \cite{MR885128,MR3059192,MR525311,MR1376344}.

\section{Concluding Remarks}
\label{sec-concluding-remarks}

\subsubsection*{Compact alphabets.} As mentioned early, if $X$ 
is compact, then it follows from the Alexandroff Theorem 
\cite[III.5.13]{MR0117523} that $rba(X)$ is equal to the set of 
all signed and finite Borel regular
countably additive measures. Therefore, in this case the Thermodynamic Formalism 
developed here is an extension of the classical one for finite 
(\cite{MR1793194,MR1085356,MR0234697,MR0390180,MR0466493}) and compact alphabets 
(\cite{MR2864625,MR3538412,MR3377291,MR3194082}).

\subsubsection*{Shift-invariant subspaces.}  If $Y\subset X$ 
is a complete and shift-invariant subset, then 
the definition of pressure and entropy can be introduced analogously as we did
for the full shift. 
Moreover, since our main results regarding the existence of equilibrium states 
are build upon the general theory of convex analysis, they 
generalize immediately for such subshifts.

\subsubsection*{Spectral radius.} 
By using similar argument as in \cite{CER17}, we can prove the following
result. For any $f\in C_{b}(X,\mathbb{R})$, there exists at least one finitely additive
probability measure $\nu_{f}$ such that 
\[
\mathscr{L}_{f}^{*}\nu_{f} = \lambda\nu_{f},
\]
where $0<\lambda\leq \rho(\mathscr{L}_{f})$. 
At this moment we do not know what are the necessary and sufficient 
conditions to ensure that $\nu_{f}$ is countably additive. 
It also seems that there $\lambda$ may not 
be the spectral radius of the Ruelle operator acting on $C_{b}(X,\mathbb{R})$.

\subsubsection*{Uniqueness.} As far as we know, the first paper proving the uniqueness of  
equilibrium states for H\"older potentials in an uncountable alphabet setting is \cite{ACR17}.
The techniques employed there are no longer applicable here, because they are
strongly dependent on the denseness of the H\"older potentials 
in the space $C_b(X,\mathbb{R})$, which may not be true if $X$ is not compact.
As mentioned before, the G\^ateaux differentiability of the pressure
would imply this result, but to the best of our knowledge none
of the known techniques can be adapted to work in the generality
considered here. 

\subsubsection*{Stone-\v{C}ech compactification.}
Due to Knowles correspondence theory developed in \cite{MR0204602}, 
there is no technical advantage in reconstructing our theory by regarding 
$X$ as a subset of its Stone-\v{C}ech compactification $\beta X$. 
To be more precise: the question whether an equilibrium state $\mu_{f}$, for 
a general potential $f\in C_{b}(X,\mathbb{R})$, is a countably
additive measure is simply translated to a question on the 
support of a corresponding measure.
For example, as an application of Theorem 2.1 of \cite{MR0204602}, it
follows that the Yosida-Hewitt decomposition of the equilibrium 
state $\mu_f$ has no 
purely finitely additive part if and only if $\overline{\mu_{f}}(Z)=0$
for every zero-set $Z$ in $\beta X$ disjoint from $X$,
see \cite{MR0204602} for more details and the definition of 
$\overline{\mu_{f}}$. 

\subsubsection*{Phase transitions.}
If we have phase transition (in the sense of multiple equilibrium states at the same temperature)
for a normalized potential $\beta f$, 
then the semigroup $\{S(t):t\geq 0\}$ 
generated by the operator $T=(\mathscr{L}_{\beta f}-I)$ is not ergodic in the
sense of \cite{MR2108619}. 
We believe that distinct cluster points in the weak-$*$-topology
of $S(t)^*(\nu)$, when $t$ tends to infinity, 
for suitable choices of $\nu$, 
will generate distinct solutions
for the infinitely dimensional stochastic differential equation 
$
dX^n_t = dB^n_t -\langle e_n, \nabla f(\sigma^{n}(X^1_t,X^2_t,\ldots))\rangle\, dt.
$
Although we do not have a rigorous argument that supports this claim, 
it seems to be at least consistent with what is known about
both problems for H\"older potentials.
Furthermore, a rigorous proof of such relation would have the potential 
of creating a beautiful bridge between Thermodynamic Formalism 
and the theory of infinite-dimensional diffusions. 


\section{Acknowledgments}
E. Silva would like to thank Jochen Gl\"uck for provide a proof of Theorem \ref{Functional}. 
This study was financed in part by the Coordena\c c\~ao de Aperfei\c coamento 
de Pessoal de N\'ivel Superior - Brasil (CAPES) - Finance Code 001. 
L. Cioletti and M. Stadlbauer would like to 
acknowledge financial support by CNPq through projects 
PQ 310883/2015-6, 310818/2015-0 and Universal 426814/2016-9, 
whereas L. Cioletti and E. Silva would like to thank FAP-DF for  financial support.


\newcommand{\etalchar}[1]{$^{#1}$}

\end{document}